\documentclass[11pt]{article}
\usepackage{amsfonts}
\usepackage{amsmath}
\usepackage{amssymb}
\usepackage{amsthm}

\usepackage[numeric]{amsrefs}

\usepackage{stmaryrd}
\usepackage[all]{xy}
\usepackage{tikz}
\usepackage{tikz-cd}
\tikzset{
    symbol/.style={%
        draw=none,
        every to/.append style={%
            edge node={node [sloped, allow upside down, auto=false]{$#1$}}}
    }
}

  \usepackage[textsize=scriptsize]{todonotes}

    \usepackage[bottom]{footmisc}

\usepackage{hyperref}
\hypersetup{
	linktocpage,
    colorlinks = True,
    citecolor=gray,
    filecolor=black,
    linkcolor=blue,
    urlcolor=blue,
    breaklinks=True
}

  \usepackage{enumitem} 

\usepackage{bbm} 

\newcommand{\B}{\mathcal{B}}
\newcommand{\C}{\mathcal{C}}

\newcommand{\V}{\mathbf{V}}
\newcommand{\D}{\mathcal{D}}
\newcommand{\F}{\mathcal{F}}
\newcommand{\M}{\mathcal{M}}
\newcommand{\K}{\mathcal{K}}

\renewcommand{\P}{\mathcal{P}}

\newcommand{\cat}[1]{\mathsf{#1}}
\newcommand{\Cat}{\cat{Cat}}
\newcommand{\Catv}{\V\textsf{-}\Cat}
\newcommand{\Set}{\cat{Set}}
\newcommand{\sSet}{\cat{sSet}}

\newcommand{\Top}{\cat{Top}}
\newcommand{\Vect}{\cat{Vect}_k}
\newcommand{\OpFib}{\cat{opFib}}
\newcommand{\Fun}{\cat{Fun}^\mathsf{ps}}
\newcommand{\I}{\mathbbm{1}}

\newcommand{\To}{\Rightarrow}
\newcommand{\Tto}{\Rrightarrow}
\newcommand{\cE}{\mathcal{E}}

\renewcommand{\1}{\mathbf{1}}

\newcommand{\lift}[1]{{#1}_!} 
\newcommand{\Inv}{I}

\theoremstyle{definition} 
\newtheorem{defi}{Definition}[section]
\newtheorem{rem}[defi]{Remark}

\newtheorem{example}[defi]{Example}

\theoremstyle{plain}
\newtheorem{thm}[defi]{Theorem}
\newtheorem*{thm*}{Theorem}
\newtheorem{lem}[defi]{Lemma}
\newtheorem{prop}[defi]{Proposition}

\begin{document}

\title{The Enriched Grothendieck Construction}
\author{Jonathan Beardsley and Liang Ze Wong}
\maketitle

\begin{abstract}
  We define and study opfibrations of $\V$-enriched categories when $\V$ is an extensive monoidal category whose unit is terminal and connected. 
  This includes sets, simplicial sets, categories, or any locally cartesian closed category with disjoint coproducts and connected unit.
  We show that for an ordinary category $B$, there is an equivalence of 2-categories between $\V$-enriched opfibrations over the free $\V$-category on $B$, and pseudofunctors from $B$ to the 2-category of $\V$-categories. 
  This generalizes the classical ($\Set$-enriched) Grothendieck correspondence. 
\end{abstract}

\tableofcontents

\section{Introduction}
The Grothendieck construction and its inverse relate stacks on a Grothendieck site to fibrations, or fibered categories, over that site. 
More generally, for any category $B$, there is an equivalence between pseudofunctors $B^{op} \to \Cat$ and fibrations over $B$ \cite[Theorem B1.3.6]{johnstone1}; by duality, there is also an equivalence between pseudofunctors $B \to \Cat$ and \emph{op}fibrations over $B$.
\[
	\OpFib(B) \cong \Fun(B, \Cat).
\]
We generalize this equivalence to categories enriched over a suitable monoidal category $\V$.
Our main result is Theorem \ref{mainthm}:
\begin{thm*}
	Let $\V$ be a monoidal category satisfying the assumptions in \S\ref{sec:grcon}, and let $B_\V$ be the free $\V$-category on a category $B$. 
	There is a $2$-equivalence
	\[
	\OpFib(B_\V) \cong \Fun(B, \Catv).
	\]	
\end{thm*}

Our paper is structured in the following manner:
In \S \ref{sec:prelims}, we briefly recall some notions from enriched category theory and $2$-category theory that will be used in the rest of the paper, leaving the details to Appendix \ref{sec:appendix}.

In \S \ref{sec:opfibinvgrcon} we define $\V$-enriched opfibrations and show that such opfibrations $p \colon \cE \to \B$ give rise to pseudofunctors $\B_0 \to \Catv$, where $\B_0$ is the underlying category of the $\V$-category $\B$. 
This gives the \emph{inverse Grothendieck construction}
\[
	I \colon \OpFib(\B) \to \Fun(\B_0, \Catv).
\]

In \S \ref{sec:grcon} we show that, under suitable assumptions on $\V$, pseudofunctors $B \to \Catv$ give rise to opfibrations over $B_\V$, the free $\V$-category on $B$. This gives the \emph{Grothendieck construction}
\[
	Gr \colon \Fun(B, \Catv) \to \OpFib(B_\V).
\]
Our hypotheses on $\V$ ensure that $(B_\V)_0 \cong B$, so that it makes sense to ask if our constructions are mutual inverses when $\B = B_\V$.
In \S \ref{sec:grcorr}, we answer this in the affirmative, yielding our main result.

We make use of standard notions and techniques from enriched category theory, and as far as possible, try to relate our constructions back to the classical Grothendieck construction (i.e.\ the $\Set$-enriched case). 
In fact, one of the features of the construction given here is that a young mathematician well versed in the definitions of enriched and 2-category theory could understand all of our proofs. 

\subsection{Motivation and relation to other work}

The original motivation for this work goes back to the PhD thesis of the first author (some of which is described in \cite{beardsrelative}). 
One of the goals of that thesis was to describe certain \emph{coalgebraic} structures (e.g.\ bialgebras and comodules) in quasicategories, which are one of several models for $(\infty,1)$-categories. 
The first author found that the $(\infty,1)$-categorical Grothendieck construction described in \cite{htt} and expanded upon in \cite{ha} were not rigid enough for these purposes. 
Recalling from \cite{bergner} that we may think of simplicially enriched categories as a model for $(\infty,1)$-categories, we may then think of the enriched Grothendieck construction given in this paper, when $\V=\sSet$, as a rigidified version of Lurie's quasicategorical Grothendieck construction over an ordinary category.
This perspective is developed further in \cite{beardsley2018operadic}.

Our inverse Grothendieck construction (\S \ref{sec:opfibinvgrcon}) is an instance of the $2$-functor $\OpFib_\K(\B) \to \Fun(\B_0, \K)$ that arises from an opfibration $\cE \to \B$ in any $2$-category $\K$ with finite $2$-limits. 
This $2$-functor is due to Ross Street, but precise references are hard to find (see \cite[\S 6]{riehl2017comprehension} for the analogous construction in an $\infty$-cosmos, from which the reader may distill the original $2$-categorical construction).
A reader who is familiar with these results may safely skip this section.

An anonymous referee has pointed out that the Grothendieck construction (\S \ref{sec:grcon}) and the Grothendieck correspondence (\S \ref{sec:grcorr}) factor through a simpler correspondence between pseudofunctors $B \to \Catv$ and opfibered $B$-\emph{parametrized} $\V$-categories. 
In other words, letting $\V_B \textsf{-}\OpFib$ be the category of opfibered $B$-parametrized $\V$-categories, there is a $2$-equivalence
\[
	\V_B\textsf{-}\OpFib \cong \Fun(B, \Catv).
\]
This result holds for arbitrary monoidal categories $\V$, and is seen to be a special case of \cite[\S 7.6]{kelly2002categories}.
A similar result for opfibered $B$-\emph{graded} $\V$-categories may be found in \cite[\S 2.4]{lowen2008hochschild}, where $\V$ is required to have coproducts which are preserved by $\otimes$.

However, these opfibered $B$-parametrized or $B$-graded categories are not \emph{functors} in $\Catv$. 
In light of these observations, our contribution in this work may alternatively be interpreted thus:~we identify properties of $\V$ under which opfibered $B$-parametrized $\V$-categories may be internalized as actual opfibrations over $B_\V$ in $\Catv$, so that we have an equivalence
\[
	\OpFib(B_\V) \cong \V_B\textsf{-} \OpFib \cong \Fun(B, \Catv).
\]

We note that there are other works dealing with various aspects of the Grothendieck construction for categories enriched over specific $\V$.
Lurie has essentially defined what an opfibration of $\sSet$-categories ought to be in \cite{htt}, although the results regarding the $\infty$-categorical Grothendieck correspondence are formulated in (marked) simplicial sets.
When $\V = \Cat$, Hermida \cite{hermida1999some}, Bakovic \cite{bakovicgrothendieck} and Buckley \cite{buckley2014fibred} give a fully enriched Grothendieck construction and Grothendieck correspondence.

By constrast, our Grothendieck construction only allows for pseudofunctors from an \emph{ordinary} category $B$ into $\Catv$, which correspond to opfibrations over the \emph{free} $\V$-category $B_\V$.
This is unavoidable at our level of generality, since $\Catv$ is not necessarily $\V$-enriched\footnote{Even if $\V$ is complete and symmetric monoidal closed, $\Catv$ is only enriched over $\Catv$, not $\V$.}, so that it does not make sense to talk about $\V$-functors $\B \to \Catv$ from an arbitrary $\V$-category $\B$ to $\Catv$. 
In addition, although we do define opfibrations over an arbitrary $\V$-enriched base $\B$, one anonymous referee has pointed out that in the case when $\V = \sSet$ or $\Cat$, our definition agrees with those given in \cite[2.4.1.10, 2.4.2.1]{htt} and \cite[2.1.6, 3.1.5]{buckley2014fibred} only when $\B = B_\V$.
Nevertheless, we believe this work is an important step towards a fully enriched Grothendieck correspondence relating opfibrations over $\B$ and $\V$-functors $\B \to \Catv$, under additional assumptions on $\V$ such that $\Catv$ is $\V$-enriched.

Finally, we note that there is already a preprint of Tamaki \cite{tamaki} which discusses an enriched Grothendieck construction.
However, the definition of an opfibration given there is equivalent to the classical definition (when $\V = \Set$) only when the base category $B$ is a groupoid.
While we have taken some inspiration from \cite{tamaki}, our work is significantly different.

\subsection{Acknowledgements}

The authors thank the two anonymous referees for their very detailed reviews and helpful suggestions, which have greatly simplified and improved this paper. 
We would also like to thank James Zhang for his support and mentorship during the completion of this paper.

\section{Preliminaries} 
\label{sec:prelims}
	We begin by recalling a few notions from enriched category theory and $2$-category theory that will be used in this paper.
	Throughout, we work over a monoidal category $(\V, \otimes, \1)$.

\subsection{Properties of $\V$} \label{sec:Vprop}
	We describe a few additional properties of $\V$ that we will later require. 
	\begin{enumerate}
	\item 	If $\V$ has coproducts, we say that \emph{$\otimes$ preserves coproducts (in both variables)} if there is a canonical isomorphism
	\[
		\bigg(\coprod_{i \in I} A_i\bigg) \otimes \bigg(\coprod_{j \in J} B_j\bigg) \cong \coprod_{i \in I} \coprod_{j \in J} A_i \otimes B_j.
	\]

	\item 	If $\V$ has pullbacks and coproducts, we say that $\V$ is \emph{extensive} if pullbacks interact well with coproducts in the following sense:
	\begin{enumerate}[label = (\roman*)]
		\item \emph{Pullbacks preserve coproduct injections}: For any set $I$ and family of maps $f_i \colon Y_i \to X_i$ in $\V$, the following square is a pullback:
		\[
			\begin{tikzcd}
				Y_i \ar[d, "f_i"'] \ar[r, hookrightarrow] & \coprod_{i \in I} Y_i \ar[d, "\coprod_i f_i"]
				\\
				X_i \ar[r, hookrightarrow] & \coprod_{i \in I} X_i
			\end{tikzcd}
		\] 
		\item \emph{Pullbacks preserve coproduct decompositions}: For any set $I$ and family of maps $f_i \colon X_i \to Z$ and $g \colon Y \to Z$ in $\V$, we have a canonical isomorphism
	\[
		Y \times_Z \left(\coprod_i X_i \right) \cong \coprod_i \left( Y \times_Z X_i \right),
	\]
	where these fibered products are given by the following pullback diagrams:
	\[
		\begin{aligned}
		\begin{tikzcd}
			Y \times_Z \left(\coprod_i X_i \right) \ar[r] \ar[d] \arrow[dr, phantom, "\lrcorner", very near start]   & Y\ar[d, "g"]
			\\
			\coprod_{i \in I} X_i \ar[r, "\coprod_i f_i"] & Z
		\end{tikzcd}
		\end{aligned}
		\quad \quad
		\begin{aligned}
		\begin{tikzcd}
			Y \times_Z X_i \arrow[dr, phantom, "\lrcorner", very near start] \ar[r] \ar[d]   & Y\ar[d, "g"]
			\\
			X_i \ar[r, "f_i"] & Z
		\end{tikzcd}
		\end{aligned}	
	\]
	\end{enumerate}

	\item 	If the monoidal product $\otimes$ is the cartesian product $\times$, we say that $\V$ is \emph{cartesian}.
	This implies that the monoidal unit $\1$ is terminal.
	If only this last condition holds, we say that $\V$ is \emph{semi}cartesian.

	\item Finally, if $\V$ is extensive, we say that the monoidal unit $\1$ is \emph{connected} if the representable functor 
	\[
		\V(\1, -) \colon \V \to \Set
	\]
	preserves all coproducts.
	If $\V$ is also semicartesian, then $\V(\1, \1) \cong \{*\}$, so for any set $X$ we have a canonical isomorphism 
	\begin{equation} \label{eq:V-prop4}
		\V\bigg(\1, \coprod_{x \in X} \1 \bigg) \cong \coprod_{x \in X} \{*\} \cong X.
	\end{equation}
	This last isomorphism is equivalent to the left adjoint of $\V(\1,-)$ (defined later in (\ref{eq:set-tensor})) being fully faithful.
	\end{enumerate}

	\begin{rem}
		Our Properties 2(i) and 2(ii) are respectively (e1) and (e$2'$) in the characterization of extensivity from \cite[\S 4.2]{centazzo97sheaf}.
		Another way of stating Property 2(ii) is that \emph{coproducts are universal}.
		Note that the definition we have used is sometimes called \emph{infinitary} extensive.
	\end{rem}

	\begin{rem}
		The above assumptions hold for the categories of sets, simplicial sets, topological spaces or categories, equipped with the cartesian monoidal product $\times$.
		They also hold for any locally cartesian closed category with disjoint coproducts and connected unit.

		An example where the above assumptions do \emph{not} hold is $(\Vect, \otimes_k, k)$.
	\end{rem}

	In \S \ref{sec:invgrcon}, we will only require that $\V$ has pullbacks, while in \S \ref{sec:grcon}, we will require all the above but with \emph{semi}cartesian instead of cartesian $\V$, and with only the isomorphism (\ref{eq:V-prop4}) of Property 4.

\subsection{Underlying categories and free $\V$-categories}
	\label{sec:underfree-prelim}

	The $2$-category of $\V$-categories, $\V$-functors and $\V$-natural transformations (defined in Appendix \ref{sec:appendix}) will be denoted $\Catv$.

	Throughout this paper, $\B$ will denote a $\V$-category with hom-objects $\B(b,c) \in \V$, while $B$ will denote an ordinary category with hom-sets $B(b,c)$.
	Let $\I$ denote the $\V$-category with a single object $*$ and
		\[
			\I(*,*) := \1.
		\]

	When $\V$ has coproducts, the representable functor $\V(\1, -)$ has a left adjoint $(-) \cdot \1 \colon \Set \to \V$ sending $X$ to
	\begin{equation} \label{eq:set-tensor}
		 X \cdot \1 := \coprod_{x \in X} \1.
	\end{equation}
	Further, if coproducts in $\V$ are preserved by $\otimes$, this adjunction between $\Set$ and $\V$ induces an adjunction:
	\begin{equation}
		\label{eq:CatVCat}
		\begin{tikzcd}[column sep = large]
			\Cat \ar[r, bend left = 20, shift left, "(-)_\V", ""{name = L}, start anchor = east, end anchor = west]
			& 
			\Catv \ar[l, bend left = 20, shift left, "(-)_0", ""{name = R}, start anchor = west, end anchor = east] \ar[from = L, to = R, symbol = \dashv]
		\end{tikzcd}
	\end{equation}	
	In detail, the \emph{underlying category} of $\B \in \Catv$ is the functor category
	\[
		\B_0 := \Catv(\I, \B)
	\]
	whose objects are $\V$-functors $b \colon \I \to \B$ and morphisms are $\V$-natural transformations $f \colon b \To c$.
	Equivalently, $\B_0$ is the category with the same objects as $\B$ and morphisms $f \colon \1 \to \B(b,c)$.

	The \emph{free $\V$-category} on  $B \in \Cat$ is the $\V$-category $B_\V$ with the same objects as $B$ and hom-objects
	\[
		B_\V(b,c) := B(b,c) \cdot \1 = \coprod_{f \in B(b,c)}\!\!\! \1.
	\]

	\begin{rem} \label{rem:BV0}
	Let $\iota$ and $\sigma$ denote the unit and counit of the adjunction $(-)_\V \dashv (-)_0$.
	Their components are 
	\begin{align*}
		\iota_B &\colon B \to (B_\V)_0, \\
		\sigma_\B &\colon (\B_0)_\V \to \B.
	\end{align*}	
	If the canonical isomorphism (\ref{eq:V-prop4}) from Property 4 holds, then $\iota_B$ is an isomorphism of categories
	\[
		B \cong (B_\V)_0.
	\]
	We elaborate on this further in Remark \ref{rem:identifyBBV0}.
	\end{rem}

	Some examples in the one-object case might be instructive.
	A one-object category may be identified with a monoid $M$, while a one-object $\V$-category may be identified with a monoid $\M$ in $\V$.

	\begin{example}
		When $\V = \Top$, the free topological monoid $M_\V$ on a monoid $M$ is the same monoid given the discrete topology, while the underlying category $\M_0$ of a topological monoid $\M$ is the same monoid forgetting its topology.
	
		In this case, $M = (M_\V)_0$, so $\iota_M$ is the identity.
		The map $\sigma_\M$ is also the identity on the underlying sets, but its domain has the discrete topology.
	\end{example}

	\begin{example}
		When $\V= \Vect$, the $k$-algebra $M_\V$ is the monoid-algebra $k[M]$, while the monoid $\M_0$ is the $k$-algebra $\M$ treated simply as a monoid (forgetting its $k$-linear structure).

		Unlike for $\Top$, in this case $M \neq k[M]$, but $\iota_M \colon M\hookrightarrow k[M]$ is the inclusion of $M$ as a basis.
		The map $\sigma_\M \colon k[\M] \to \M$ sends formal linear combinations of objects in $\M$ to their actual sum in $\M$.
	\end{example}

\subsection{Pseudofunctors, natural transformations, modifications}
	\label{sec:pseudoFTM}
	We will also be working with \emph{pseudofunctors} $F \colon B \to \Catv$ out of an ordinary category $B$.
	These are `functors' that are only associative and unital up to coherent isomorphism.
	More precisely, $F$ consists of the following data,

\begin{enumerate}
	\item for each $b \in B$, a $\V$-category $F_b$;
	\item for each $f \colon  b \to c$, a $\V$-functor $F_f \colon F_b \to F_b$;
	\item for each $b \in B$, a natural isomorphism $\xi(b) \colon F_{1_b} \cong 1_{F_b}$ (or simply $\xi$) with components:
	\[
		\xi_x \colon \1 \to F_b(F_{1_b} x, x)
	\]
	\item for each $b \xrightarrow{f} c \xrightarrow{g} d$ in $B$, a natural isomorphism $\theta(f,g) \colon F_{gf} \cong F_g F_f$ (or simply $\theta$) with components:
	\[
		\theta_x \colon \1 \to F_d(F_{g f} x , F_g F_f x)
	\]	
\end{enumerate}
satisfying the following relations:
	\begin{equation} \label{eq:pseudo-unit-left}
		\begin{aligned}
		\begin{tikzcd}
			& F_b \ar[dr, "F_f"]  &
			\\
			F_b \ar[ur, ""{name = D, above}, bend right = 20] \ar[rr, bend right = 10, "F_f"',""{name = L, above}] \ar[ur, bend left = 40, equals, ""{name = C, below}] \ar[from = C, to = D, Leftarrow, "\xi"] &  & F_d \ar[from = 1-2, to = L, Leftarrow, "\theta", pos=0.6, shorten <=0.8em, shorten >=0.5em]
		\end{tikzcd}
		\end{aligned}
		\quad = \quad
		\begin{aligned}
		\begin{tikzcd}
			& \phantom{F_d} & 
			\\
			F_b \ar[rr, bend left=20, "F_f", ""{name = U, below}] \ar[rr, bend right=20, "F_f"', ""{name = L, above}] & & F_d \ar[from = U, to = L, Leftarrow, "=", pos=0.6, shorten <=0.2em, shorten >=0.2em]
		\end{tikzcd}	
		\end{aligned}
	\end{equation}
	\begin{equation} \label{eq:pseudo-unit-right}
		\begin{aligned}
		\begin{tikzcd}
			& F_d \ar[dr,  ""{name = E, above}, bend right = 20] \ar[dr, bend left = 40, equals, ""{name = D, below}] \ar[from = D, to = E, Leftarrow, "\xi"]  &
			\\
			F_b \ar[ur, "F_f"] \ar[rr, bend right=10, "F_f"',""{name = L, above}]  &  & F_d \ar[from = 1-2, to = L, Leftarrow, "\theta"', pos=0.6, shorten <=0.8em, shorten >=0.5em, ]
		\end{tikzcd}
		\end{aligned}
		\quad = \quad
		\begin{aligned}
		\begin{tikzcd}
			& \phantom{F_d} & 
			\\
			F_b \ar[rr, bend left=20, "F_f", ""{name = U, below}] \ar[rr, bend right=20, "F_f"', ""{name = L, above}] & & F_d \ar[from = U, to = L, Leftarrow, "=", pos = 0.6, shorten <=0.2em, shorten >=0.2em]
		\end{tikzcd}	
		\end{aligned}
	\end{equation}
	\begin{equation} \label{eq:pseudo-assoc}
		\begin{aligned}
		\begin{tikzcd}
				&[-15pt] F_b \ar[r, "F_g"] \ar[drr,, ""{name = L, left}, bend right = 10] & F_d \ar[dr, "F_h"] \ar[to = L, Leftarrow, "\theta"', pos = 0.9] &[-15pt]
			\\
			F_b \ar[ur, "F_f"] \ar[rrr, "F_{hgf}"', bend right = 10, ""{name = B, left}, pos=0.4 ] & \phantom{F_b} & & F_e \ar[from  =1-3, to = B, Leftarrow, "\theta"', pos=0.7, shorten <=1.8em, shorten >=0.5em]
		\end{tikzcd}
		\end{aligned}
		\quad = \quad
		\begin{aligned}
		\begin{tikzcd}
				&[-15pt] F_b \ar[r, "F_g"] & F_d \ar[dr, "F_h"] &[-15pt]
			\\
			F_b \ar[ur, "F_f"] \ar[rrr, "F_{hgf}"', bend right = 10, ""{name = B,right}, pos = 0.6] \ar[urr, , ""{name = L, right}, bend right = 10] &  & & F_e \ar[from  =1-2, to= L, Leftarrow, "\theta", pos = 0.9] \ar[from = 1-2, to = B, Leftarrow, "\theta", pos=0.7, shorten <=1.8em, shorten >=0.5em]
		\end{tikzcd}
		\end{aligned}
	\end{equation}

	A \emph{pseudonatural transformation} (or simply a \emph{transformation}) $\alpha \colon F \To G$ between pseudofunctors consists of $1$-cells $\alpha_b \colon F_b \to G_b$ for each $b \in B$ and natural isomorphisms
	\begin{equation} \label{eq:trans}
		\begin{tikzcd}
			F_b \ar[r, "F_f"]\ar[d, "\alpha_b", swap] & F_b \ar[d, "\alpha_b"]\\
			G_b \ar[r, "G_f", swap] \ar[ur, Rightarrow,"\alpha_f", "\cong"', start anchor={north east}, end anchor={south west}, shorten <=0.7em, shorten >=0.7em]& G_b
		\end{tikzcd}
	\end{equation}
	for each $f \colon b \to c$ in $B$, satisfying further coherence rules given in \cite[\S 1.2]{leinsterbicats}.

	Finally, a  \emph{modification} $\Gamma\colon \alpha\Tto \beta$ between transformations consists of natural transformations $\Gamma_b \colon \alpha_b \To \beta_b$ for each $b \in B$ satisfying:
	\begin{equation} \label{eq:mod}	
	\begin{aligned}
		\begin{tikzcd}[column sep = huge, row sep = huge]
			F_b \ar[r, "F_f"]\ar[d, "\alpha_b" description, bend right = 40]  & F_b \ar[d, "\alpha_b" description, bend right = 40, ""{name = L}] \ar[d, "\beta_b" description, bend left = 40, ""{name = R, left}] \ar[from  = L, to = R, Rightarrow, "\Gamma_b"]\\
			G_b \ar[r, "G_f", swap] \ar[ur, Rightarrow,"\alpha_f" near start, start anchor={north east}, end anchor={south west}, shorten <=0.5em, shorten >=3.5em, "\cong"' near start]& G_b
		\end{tikzcd}
	\end{aligned}
	\quad = \quad
	\begin{aligned}
		\begin{tikzcd}[column sep = huge, row sep = huge]
			F_b \ar[r, "F_f"]\ar[d, "\alpha_b" description, bend right = 40, ""{name = L}] \ar[d, "\beta_b" description, bend left = 40, ""{name = R, left}] \ar[from = L, to = R, Rightarrow, "\Gamma_b"] & F_b \ar[d, "\beta_b" description, bend left = 40] \\
			G_b \ar[r, "G_f", swap] \ar[ur, Rightarrow,"\beta_f" near end, start anchor={north east}, end anchor={south west}, shorten <=3.5em, shorten >=0.5em, "\cong"' near end]& G_b
		\end{tikzcd}
	\end{aligned}	
	\end{equation}

	Let $\Fun(B, \Catv)$ denote the $2$-category of pseudofunctors, transformations and modifications from $B$ to $\Catv$.

\section{Opfibrations and the Inverse Grothendieck Construction}
	\label{sec:opfibinvgrcon} 
	In this section, we develop the theory of opfibrations in the enriched setting.	
	We define opfibrations over a base $\B$, and the 2-category $\OpFib(\B)$ that they form.
	The inverse Grothendieck construction is then a 2-functor from $\OpFib(\B)$ to the category of pseudofunctors $\Fun(\B_0, \Catv)$.

	Throughout we will assume that $\V$, and hence $\Catv$, has pullbacks.

\subsection{Opfibrations and opfibered functors}

\begin{defi}  \label{def:opcartesian}
	Let $p \colon \cE \to \B$ be a $\V$-functor.
	A map $\chi \colon \1 \to \cE(e, e')$ 	is \textbf{$p$-opcartesian} if the following square is a pullback in $\V$ for all $d \in \cE$:
	\begin{equation} \label{eq:opcartesian}
		\begin{tikzcd}[column sep = large, ]
			\cE(e', d) \ar[d, "p"'] \ar[r, "-\circ \chi"] & \cE(e,d) \ar[d, "p"]
			\\
			\B(pe', pd) \ar[r, "-\circ p\chi"] & \B(pe,pd)
		\end{tikzcd}
	\end{equation}
\end{defi}

\begin{defi}
	\label{def:opfib}
	An \textbf{opfibration} is a $\V$-functor $p\colon \cE \to \B$ along with, for every $e \in \cE$, $b \in \B$ and $f \colon \1 \to \B(pe, b)$, 
	an object $\lift{f}e \in \cE$ over $b$ and a $p$-opcartesian map $\chi(f,e) \colon \1 \to \cE(e, \lift{f}e)$ over $f$.

	The map $\chi(f,e)$ will be called a \textbf{chosen} $p$-opcartesian lift of $f$.
\end{defi}

We note a simple but important lemma:

\begin{lem} \label{lem:iso-to-chosen}
	Let $p \colon \cE \to \B$ be an opfibration.
	A map $\chi \colon \1 \to \cE(e,e')$ is $p$-opcartesian if and only if it is isomorphic to the chosen $p$-opcartesian lift $\chi(p\chi, e) \colon \1 \to \cE(e, \lift{(p\chi)} e)$ via a unique isomorphism 
	\begin{equation} \label{eq:eps-chi}
		\varepsilon_\chi \colon \1 \to \cE(\lift{(p\chi)}e, e')
	\end{equation}
	lying over $1_{e'}$.
\end{lem}
\begin{proof}
	This follows from the uniqueness (up to unique isomorphism) of the pullback in Definition \ref{def:opcartesian}.
\end{proof}

\begin{defi}
	An \textbf{opfibered functor} from $p\colon \cE \to \B$ to $q \colon \F \to \B$ is a functor $k \colon \cE \to \F$ that satisfies $qk = p$ and sends $p$-opcartesian maps to $q$-opcartesian maps.
\end{defi}

\begin{lem} \label{lem:opfibered-cocart}
	Let $p \colon \cE \to \B$ and $q \colon \F \to \B$ be opfibrations over $\B$, and let $k \colon \cE \to \F$ be such that $qk = p$.
	Then $k$ is opfibered if and only if it sends chosen $p$-opcartesian maps to (not necessarily chosen) $q$-opcartesian maps. 
\end{lem}
\begin{proof}
	If $k$ is opfibered, it certainly sends chosen $p$-opcartesian maps to $q$-opcartesian ones.

	Conversely, suppose $k$ sends chosen $p$-opcartesian maps to $q$-opcartesian ones. 
	Lemma \ref{lem:iso-to-chosen} then shows that $k$ is opfibered: any $p$-opcartesian map $\chi \colon \1 \to \cE(e,e')$ is isomorphic to the chosen $p$-opcartesian map $\chi(p\chi, e)$.
	Since functors preserve isomorphisms, $k\chi$ is isomorphic to the $q$-opcartesian map $k \chi(p\chi, e)$, hence is also $q$-opcartesian.
\end{proof}

\begin{defi}
	Let $\OpFib(\B)$ denote the $2$-category whose objects are opfibrations over $\B$, morphisms are opfibered functors, and 2-morphisms are natural transformations over $\B$.
\end{defi}

\subsection{Properties of opcartesian maps}

	In this section, we record some results concerning opcartesian maps and the unique maps that their universal properties induce. 

	Let $p \colon \cE \to \B$ be an opfibration.
	By the universal property of $p$-opcartesian maps, a pair of maps $\varphi \colon \1 \to \cE(e,d)$  and $g \colon \1 \to \B(b, pd)$ such that $p\varphi = gf$ induces a unique $\widetilde{g}$:
	\[
		\begin{tikzcd}[sep = large]
			\1 \ar[drr, "\forall \,\varphi", bend left = 13] \ar[ddr, "\forall \, g"', bend right] \ar[dr, dashed, "\exists !\,\widetilde{g}" description]
			\\
			& \cE(\lift{f}e, d) \ar[r, "{-\circ \chi(f,e)}"] \ar[d, "p"']  \ar[dr,phantom, "\lrcorner", very near start] & \cE(e,d) \ar[d, "p"]
			\\
			& \B(b, pd) \ar[r, "-\circ f"] & \B(pe,pd)			
			\end{tikzcd}
	\]
	When $\V = \Set$, the preceding discussion yields the universal property:
	for every $\varphi \colon e \to d$ and $g \colon b \to pd$  such that $p\varphi = gf$, there exists a unique $\widetilde{g} \colon \lift{f} e \to d$ such that $p \widetilde{g} = g$ and $\widetilde{g}\chi(f,e) = \varphi$:
	\begin{equation}
		\label{eq:opcart-classical}
		\begin{tikzcd}[column sep = large, row sep =large]
			e \ar[r, "{\chi(f,e)}"] \ar[d, dotted] & \lift{f}e \ar[dr, dashed, "\exists !\, \widetilde{g}"] \ar[d, dotted] &
			\\
			pe \ar[r, "f"] \ar[drr, bend right = 10, "gf"'] & b \ar[dr, "\forall\, g"] & d \ar[d, dotted] \ar[from = 1-1, bend right = 10, crossing over, "\forall\, \varphi", near start]
			\\
			& & pd
		\end{tikzcd}
	\end{equation}
	Here, the dotted arrows represent $p$, and indicate which objects and arrows of $\cE$ lie over which objects and arrows of $\B$.

	\begin{lem}	\label{lem:g-tilde-opcart}
		 $\varphi$ is $p$-opcartesian if and only if $\widetilde{g}$ is.
	\end{lem}
	\begin{proof}
		Let $d' \in \cE$.
		The various maps involved fit into the following diagram:
		\begin{equation} \label{eq:g-tilde-opcart}
			\begin{tikzcd}[sep = large]
				\cE(d, d') \ar[r, "-\circ \widetilde{g}"] \ar[rr, bend left, "-\circ \varphi"] \ar[d, "p"]
				&
				\cE(\lift{f}e, d') \ar[r, "-\circ {\chi(f,e)}"] \ar[d, "p"] \ar[dr,phantom, "\lrcorner", very near start]
				&
				\cE(e,d') \ar[d, "p"]
				\\
				\B(pd,pd') \ar[r, "-\circ g"]
				&
				\B(b, pd') \ar[r, "- \circ f"]
				&
				\B(pe, pd')
			\end{tikzcd}
		\end{equation}
		Since $\chi(f,e)$ is $p$-opcartesian, the square on the right is a pullback.
		By the pasting law for pullbacks, the outer square is a pullback ($\varphi$ is $p$-opcartesian) if and only if the square on the left is a pullback ($\widetilde{g}$ is $p$-opcartesian).
	\end{proof}

	\begin{lem} \label{lem:opcart-compose}
		A composite of $p$-opcartesian maps is $p$-opcartesian.
	\end{lem}
	\begin{proof}
		This follows from arguments similar to the previous proof.
		The relevant diagram is (\ref{eq:g-tilde-opcart}), with $\tilde{g}$ taken to be $p$-opcartesian and $\chi(f,e)$ replaced by an arbitrary $p$-opcartesian map.
	\end{proof}

	\begin{lem} \label{lem:counit-iso}
		For each $e \in \cE$, the lift $\chi(1_{pe},e)$ is an isomorphism.
		Thus
		\[
			e \cong \lift{(1_{pe})} e.
		\]
	\end{lem}
	\begin{proof}
		It is easy to see that $1_e$ is $p$-opcartesian.
		By Lemma \ref{lem:iso-to-chosen}, there is a unique $\varepsilon_{(1_e)}$ which is a left inverse to $\chi(1_{pe},e)$:
		\begin{equation} \label{eq:eps-e}
			\varepsilon_{(1_e)} \circ \chi(1_{pe},e) = 1_e.
		\end{equation}
		Since $\varepsilon_{(1_e)}$ is an isomorphism, it is a right inverse as well, so $\chi(1_{pe},e)$ is an isomorphism.
	\end{proof}

\subsection{Fibers and transport}
\label{sec:fibers}
	In this section, we define and study the fibers of an opfibration, and show that arrows in the base $\B_0$ induce transport functors between fibers. 
	We begin with a more general definition of fibers of any functor.
	\begin{defi}
		Let $p\colon  \cE \to \B$ be a $\V$-functor. For each $b \in \B$, treated as a functor $b \colon \I \to \B$, the \textbf{fiber} of $p$ over $b$ is the category $\cE_b$ given by the pullback:
		\[
			\begin{tikzcd}
				\cE_b \ar[dr,phantom, "\lrcorner", very near start] \ar[r, hookrightarrow] \ar[d] & \cE \ar[d, "p"]
				\\
				\I \ar[r, "b"'] & \B
			\end{tikzcd}
		\]
	\end{defi}	

	The objects of $\cE_b$ are $\left\{e \in \cE \, | \, pe = b \right\}$, 
	while the morphisms are given by the pullback:
	\begin{equation} 
		\label{eq:fiber-homs}
		\begin{tikzcd}
			\cE_b(e,e') \ar[r] \ar[d] \arrow[dr, phantom, "\lrcorner", very near start] & \cE(e,e') \ar[d 	, "p"]
			\\
			\1 \ar[r, "1_b"'] & \B(b,b)
		\end{tikzcd}.
	\end{equation}

	\begin{rem}
	We may think of $\cE_b$ as the subcategory of $\cE$ consisting of objects in the pre-image of $b$ and morphisms in the pre-image of $1_b$.
	\end{rem}

	\begin{prop} \label{prop:transport}
		Let $p \colon \cE \to \B$ be a opfibration.
		For every $f \in \B_0(b,b')$, the assignment $e \mapsto \lift{f}e$ extends to a functor
		\[
			\lift{f} \colon \cE_b \to \cE_{b'},
		\]
		called the \textup{\textbf{transport along $f$}}.
	\end{prop}
	\begin{proof}
		On morphisms, take $(\lift{f})_{e,e'} \colon \cE_b(e,e') \to \cE_{b'}(\lift{f}e,\lift{f}e')$ to be the unique map induced by the commuting diagram
		\[
			\begin{tikzcd}[sep = large]
				\cE_b(e,e') \ar[r, hookrightarrow] \ar[d] \arrow[dr, phantom, "\lrcorner", very near start]  & \cE(e,e') \ar[d, "p"] \ar[r, "\chi{(f, e')}\circ -"] & \cE(e, \lift{f}e') \ar[d, "p"]
				\\
				\1 \ar[r, "1_b"] & \B(b,b) \ar[r, "f \circ -"] & \B(b, b')
			\end{tikzcd}
		\]
		and the universal property of the composite pullback:
		\[
			\begin{tikzcd}[sep = large]
				\cE_{b'}(\lift{f}e,\lift{f}e') \ar[r, hookrightarrow] \ar[d] \arrow[dr, phantom, "\lrcorner", very near start]  & \cE(\lift{f}e,\lift{f}e') \ar[d, "p"] \ar[r, "-\circ \chi{(f, e)}"] \arrow[dr, phantom, "\lrcorner", very near start]  & \cE(e, \lift{f}e') \ar[d, "p"]
				\\
				\1 \ar[r, "1_{b'}"] & \B(b',b') \ar[r, "-\circ f"] & \B(b, b')
			\end{tikzcd}
		\]		
		Functoriality of $\lift{f}$ follows from the uniqueness of each $(\lift{f})_{e,e'}$.
		We leave it to the reader to check the details.
	\end{proof}

	We may attempt to define a functor
	\begin{align}
		\cE_\bullet \colon  \B_0 &\xrightarrow{?} \Catv  \nonumber \\ 
				 b &\mapsto \cE_b \label{eq:Fpseudo} \\ 
				 (b \xrightarrow{f} b') &\mapsto (\cE_{b} \xrightarrow{\lift{f}} \cE_{b'}).  \nonumber
	\end{align}
	Unfortunately, $\cE_\bullet$ fails to be a functor, as it only preserves identities and composites \emph{up to isomorphism}, i.e.\ $\cE_\bullet$ is a \emph{pseudo}functor.
	This is the subject of the next subsection.

\subsection{The Inverse Grothendieck Construction $I$}\label{sec:invgrcon}

	\begin{prop} 
	The map $\cE_\bullet \colon \B_0 \to \Catv$ can be given the structure of a pseudofunctor.
	\end{prop}
	\begin{proof}
		We need to supply natural isomorphisms $\xi$ and $\theta$ as per \S \ref{sec:pseudoFTM}.

		For each $b \in \B_0$ and $e \in \cE_b$, take $\xi_e$ to be the isomorphism
		$\varepsilon_{(1_e)}$ from Lemma \ref{lem:counit-iso}.
		For $f \in \B_0(b,c)$ and $g \in \B_0(c,d)$, we need an isomorphism
		\[
			\theta_e \colon \1 \to \cE_{d}( \lift{gf}e, \lift{g}\lift{f} e).
		\]
		By Lemma \ref{lem:opcart-compose}, the composite
		$
					\chi(g,\lift{f}e) \, \chi(f,e)
		$
		from  $e$ to $\lift{g}\lift{f} e$ is $p$-opcartesian.
		We may thus take $\theta_e$ to be the unique isomorphism
		$
				\varepsilon_{\chi(g,\lift{f}e) \chi(f,e)}
		$
		from Lemma \ref{lem:iso-to-chosen} between $\chi(gf,e)$ and $\chi(g, \lift{f}e)\, \chi(f,e)$.

		The uniqueness of these components may be used to show that they satisfy equations (\ref{eq:pseudo-unit-left}, \ref{eq:pseudo-unit-right}, \ref{eq:pseudo-assoc}), so that we do indeed obtain a pseudofunctor.
	\end{proof}

	\begin{prop}
		The inverse Grothendieck construction that sends an opfibration $p \colon \cE \to \B$ to the pseudofunctor $\cE_\bullet$ extends to a $2$-functor $$\Inv \colon \OpFib(\B) \to \Fun(\B_0,\Catv).$$
	\end{prop}
	\begin{proof}
		We have seen above what $\Inv$ does to opfibrations i.e.\ the 0-cells of $\OpFib(\B)$. We need to define what $\Inv$ does to 1-cells and 2-cells.
		Let $p \colon \cE \to \B$ and $q \colon \F \to \B$ be opfibrations over $\B$, and let $\cE_\bullet, \F_\bullet \colon \B_0 \to \Catv$ be the corresponding pseudofunctors.	

		Suppose we have an opfibered functor $k \colon \cE \to \F$.
		Pulling this back along each $b$ yields functors $k_b \colon \cE_b \to \F_b$ sending $e$ to $ke$. 
		We seek a natural isomorphism of the form
		\begin{equation} \label{eq:alpha-f}
			\begin{tikzcd}
				\cE_b \ar[d, "k_b"'] \ar[r, "\lift{f}"] & 	\cE_{b'} \ar[d, "k_{b'}"] 		
				\\
				\F_b \ar[r, "\lift{f}"']
				\ar[ur, Rightarrow, shorten >=1em, shorten <=1em, "\cong"', "\alpha_f"]	& \F_{b'}
			\end{tikzcd}		
		\end{equation}
		with components $(\alpha_f)_e \colon \1 \to \F_{b'}(\lift{f} \, k e, k \, \lift{f}e )$.
		Since $k$ is opfibered, the map $k \chi(f,e)$ is $q$-opcartesian.
		By Lemma \ref{lem:iso-to-chosen}, we may take $(\alpha_f)_e$ to be the unique isomorphism $\varepsilon_{k \chi(f,e)}$ between $\chi(f, ke)$ and $k \chi(f,e)$.

		The functors $k_b$ thus form the 1-components of a pseudonatural transformation which we denote $\kappa \colon \cE_\bullet \To \F_\bullet$, while the 2-components are given by the natural isomorphism above.

		Given another opfibered functor $h \colon \cE \to \F$ which induces $\lambda \colon \cE_\bullet \To \F_\bullet$, and a $\V$-natural transformation $\gamma \colon h \To k$ over $\B$, we may once again pull all these back along $b$ to obtain a $\V$-natural transformation
		\[
			\begin{tikzcd}
				\cE_b  \ar[r, bend left, start anchor = north east, end anchor = north west, "h_b", ""{name = U, below}] \ar[r, bend right, start anchor = south east, end anchor = south west, "k_b"', ""{name=D, above}] &  \F_b \ar[from = U, to = D, Rightarrow, "\gamma_b"]
			\end{tikzcd}
		\]
		The 2-cells $\gamma_b$ then form the data of a modification $\Gamma \colon \lambda \Rrightarrow \kappa$.

		We leave it to the reader to check that the various coherence conditions are satisfied, and that this indeed yields a $2$-functor.
	\end{proof}

	Note that while the inverse Grothendieck construction takes opfibrations over an arbitrary enriched $\V$-category $\B$, it only returns pseudofunctors from an unenriched $\B_0$.
	It is thus generally not possible to recover an opfibration $p \colon \cE \to \B$ over an arbitrary base $\B$ from its corresponding pseudofunctor $\cE_\bullet \colon \B_0 \to \Catv$.
	The next section describes the best we can do.

\section{The Grothendieck Construction}\label{sec:grcon}
We now describe an opfibration $p \colon GrF \to B_\V$ associated to a pseudofunctor $F \colon B \to \Catv$. 

\subsection{Assumptions}\label{sec:assumptions}
For what follows, we require the following assumptions described in \S \ref{sec:Vprop}:
\begin{enumerate}
	\item $\V$ has coproducts, and $\otimes$ preserves coproducts in both variables;
	\item $\V$ has pullbacks and is extensive (pullbacks preserve coproduct injections and decompositions);
	\item $\V$ is semi-cartesian ($\1$ is terminal);		
	\item For $X$ a set, we have a canonical isomorphism
	\[
		\V\bigg(\1, \coprod_{x \in X} \1 \bigg) \cong X,
	\]
	so that $B \cong (B_\V)_0$ (this holds if $\1$ is connected).
\end{enumerate}

	We briefly sketch where these properties will be used:~Coproducts are required for the formation of $B_\V$ and $GrF$, and we require $\otimes$ to commute with coproducts to define composition in these categories.
	While $\1$ need not be terminal to obtain $GrF$, we do need it to obtain a functor $p \colon GrF \to B_\V$.
	Pullbacks are needed in the very definition of an opfibration.
	Finally, extensivity and the last condition are required in order for $p$ to be an opfibration (see Remark \ref{rem:assumption4used}).

\begin{rem} \label{rem:identifyBBV0}
Property 4 implies that the unit of the adjunction (\ref{eq:CatVCat}) is an isomorphism of categories: 
\[
	B \cong (B_\V)_0.
\]
We simplify matters by assuming that this isomorphism is in fact \emph{equality}, which we justify in the following manner:

Recall that elements of $(B_\V)_0$ are $\V$-maps
\[
	\1 \to B_\V(b,c) = \coprod_{f \in B(b,c)} \1.
\]
The set of such maps contains the inclusions $\1_g \hookrightarrow B_\V(b,c)$, where $\1_g$ denotes the copy of $\1$ corresponding to $g \in B(b,c)$. 
Assumption 4 then says that these inclusions account for \emph{all} $\V$-maps $\1 \to B_\V(b,c)$.
By abuse of notation, we may \emph{identify} elements $g \in B(b,c)$ with maps
\[
	\1 = \1_g \hookrightarrow B_\V(b,c),
\]
which we also call $g$. 
Under this identification, we then have $B = (B_\V)_0$.
\end{rem}

\subsection{The category $Gr F$}

\begin{defi}\label{defi:gc}
Let $B$ be an ordinary (i.e.\ $\Set$-enriched) category treated as a $2$-category, and let $F\colon B \to \Catv$ be a pseudofunctor.
The \textbf{Grothendieck construction of $F$} is the $\V$-category $Gr F$ with objects and morphisms
\begin{align*}
		Ob(GrF) &:= \;\; \coprod_{b\in B} \;\; Ob(F_b) \times \{b\}, \\
		GrF\big( (x,b), (y,c) \big) & :=\coprod_{f\colon b\to c} F_{c}(F_f x,y).
\end{align*}
Identity morphisms are given by
\begin{equation}\label{eq:grid}
	1_{(x,b)} := \xi_x \colon \1 \to F_b(F_{1_b}x,x) \subset \coprod_{f \colon b \to b} F_b(F_f x, x) = GrF\big( (x,b), (x,b)\big)
\end{equation}
while composition is induced by the composite
\[
	\begin{tikzcd}[column sep = large]
		F_b(F_f x,y)\otimes F_d(F_g y,z) \ar[r, "F_g\otimes 1"] \ar[d, dotted] & F_d(F_g F_f x,F_g y)\otimes F_d(F_g y,z) \ar[d, "\cong", "(-\circ \theta_x)\otimes 1"']
		\\
		F_d(F_{gf}x,z)   & F_d(F_{gf}x,F_g y)\otimes F_d(F_g y,z) \ar[l, "\circ"]
	\end{tikzcd}
\]
where $b \xrightarrow{f} c \xrightarrow{g} d$.
This extends to a functor out of $GrF\big( (x,b), (y,c) \big) \otimes GrF\big( (y,c), (z,d)\big)$ because $\otimes$ preserves coproducts.

\end{defi}

\begin{rem}
We will see in what follows that the $GrF$ admits an opfibration to the free $\V$-category $B_\V$ on $B$. 
If instead we had $F \colon B^{op} \to \Catv$, we could similarly define the $\V$-category $Gr^\vee F$ where
\begin{align*}
		Ob(Gr^\vee F) &:= \;\; \coprod_{b\in B} \;\; Ob(F_b) \times \{b\}, \\
		Gr^\vee F\big( (x,b), (y,c) \big) & :=\coprod_{f\colon b\to c} F_b(x, F_f y),
\end{align*}
and show that $Gr^\vee F$ admits a \emph{fibration} to $B_\V$.
The properties of $Gr^\vee F$ are formally dual to $Gr F$. 
\end{rem}

We next produce the functor that we want to show is an opfibration.

\begin{lem}
	Let $F \colon B \to \Catv$ be a pseudofunctor, and $GrF$ its Grothendieck construction. 
	There is a $\V$-functor $p \colon Gr F\to B_\V$.
\end{lem}
\begin{proof}
	We will describe the functor but leave checking of the necessary coherences to the reader. 
	On objects, $p$ simply projects down to $B$, sending $(x,b)$ to $b$. 
	On morphisms, we need to give a $\V$-morphism 
	\[
		GrF\big( (x,b), (y,c)\big) = \coprod_{f \colon b \to c} F_b(F_f x,y)  \to \coprod_{f\colon b\to c}\1 = B_\V(b,c).
	\]
	Since $\V$ is semi-cartesian, each $F_b(F_f x,y)$ has a unique map to $\1$.
	Taking the coproduct of these maps over $B(b,c)$, we obtain the desired morphism.
\end{proof}

\subsection{The Grothendieck construction $Gr$}

\begin{prop}\label{prop:grisopfib}
	The functor $p \colon Gr F \to B_\V$ is an opfibration.
\end{prop}
\begin{proof}
	For all $(x,b) \in Gr F$ and $f \colon \1 \to B_\V(b,c)$, we need $\lift{f} (x,b) \in Gr F$ and a $p$-opcartesian lift $\chi\big(f, (x,b)\big) \colon \1 \to Gr F\big( (x,b), \lift{f} (x,b) \big)$.

	Since we have assumed that $(B_\V)_0 = B$, such an $f$ is precisely a map $f \colon b \to c$ in $B$.
	We may thus let $\lift{f}(x,b) := (F_f x, c)$ and take $\chi\big(f, (x,b)\big)$ to be the identity of $F_f x$:
	\begin{equation} \label{eq:Gr-chi}
		\chi\big(f, (x,b)\big) \colon \1 \xrightarrow{1_{F_f x}} F_b(F_f x, F_f x ) \hookrightarrow Gr F\big( (x,b), \lift{f}(x,b) \big).
	\end{equation}
	For every $(y,d) \in Gr F$, we need the following diagram to be a pullback, where we have written $\chi$ for $\chi\big( f, (x,b) \big)$:
	\[
		\begin{tikzcd}[column sep = large, ]
			Gr F \big( \lift{f}(x,b), (y,d) \big) \ar[d, "p"'] \ar[r, "-\circ \chi"] & Gr F\big( (x,b) , (y,d)\big) \ar[d, "p"]
			\\
			B_\V(c, d) \ar[r, "-\circ f"] & B_\V(b,d)
		\end{tikzcd}
	\]
	But by definition of the various objects involved and the pseudofunctoriality of $F$, the diagram above is equivalent to the diagram below, which is a pullback diagram since $\V$ is extensive:
	\[
		\begin{tikzcd}[column sep = large, ]
			\coprod\limits_{g \colon c \to d} F_d(F_{gf} x, y) 
			\ar[d, "p"'] \ar[r, ] & 
			\coprod\limits_{h \colon b \to d} F_d(F_h x, y) 
			\ar[d, "p"]
			\\
			\coprod\limits_{g \colon c \to d} \1 \ar[r, "-\circ f"] & \coprod\limits_{h \colon b \to d} \1
		\end{tikzcd}
	\]
	In more detail, Property 2(i) ensures that the following is a pullback, where the horizontal arrows are given by re-indexing along $f$ (i.e.\ precomposition with $f$):
	\[
		\begin{tikzcd}[column sep = large, ]
			F_d(F_{gf} x, y) 
			\ar[d, "p"'] \ar[r, hookrightarrow] & 
			\coprod\limits_{h \colon b \to d} F_d(F_h x, y) 
			\ar[d, "p"]
			\\
			\1_g \ar[r, hookrightarrow, "- \circ f"] & \coprod\limits_{h \colon b \to d} \1
		\end{tikzcd}
	\]	
	Property 2(ii) then implies that the diagram above it is also a pullback.
\end{proof}

 	\begin{rem} \label{rem:assumption4used}
 		Note that we only have functors $F_f$ for $f \in (B_\V)_0$ arising from actual maps in $B$.
 		This is the reason for requiring $(B_\V)_0 = B$ .
 	\end{rem}
 	 	
 	\begin{thm}
 		The construction $F\mapsto GrF$ extends to a 2-functor 
 		\[
 			Gr \colon \Fun(B,\Catv)\to \OpFib(B_\V).
		\] 
 	\end{thm}
 	
 	\begin{proof}
 		From Proposition \ref{prop:grisopfib} we have that the construction takes pseudofunctors to opfibrations. 
 		It remains to show that it takes pseudonatural transformations to opfibered functors, and modifications to natural transformations of opfibered functors.

 		Let $\alpha\colon F\Rightarrow G$ be a pseudonatural transformation between pseudofunctors $F,G \colon B \to \Catv$. 
 		In particular, for every $b\in B$ we have a $\V$-functor $\alpha_b\colon F_b\to G_b$ such that for any $f\colon b\to c$ in $B$ we have an \emph{invertible} 2-cell $\alpha_f \colon G_f\alpha_{b} \cong \alpha_b F_f$ with components
 		\begin{equation} \label{eq:alphaGFiso}
 			(\alpha_{f})_x \colon \1 \to G_b(G_f \alpha_b x , \alpha_b F_f x).
		\end{equation}
		Define an opfibered functor $a \colon GrF \to Gr G$ as follows:
 		on objects,
 		\[
 			a (x,b) := (\alpha_b x, b).
 		\]
 		On morphisms, take the composite:
 		\begin{equation} \label{eq:a-on-maps}
 			\begin{tikzcd}[row sep = small]
 				GrF\big( (x,b),  (y,c)\big) \ar[rr, dotted, "a_{(x,b),(y,c)}"] \ar[d, equals] & & GrG\big( (\alpha_b x, b), (\alpha_b y, c) \big) \ar[d, equals]
 				\\
				\coprod\limits_{f\colon b\to c} F_b(F_f x,y) \ar[r, "\alpha_b"]
				&
				\coprod\limits_{f\colon b\to c} G_b(\alpha_bF_f x,\alpha_b y) \ar[r, "-\circ (\alpha_{f})_x" , "\cong"']
				&
				\coprod\limits_{f\colon b\to c} G_b(G_f \alpha_b x,\alpha_b y)		
 			\end{tikzcd}
 		\end{equation}
 		Pseudofunctoriality of $F$ and $G$, along with pseudonaturality of $\alpha$, immediately indicate that these data assemble into a $\V$-functor $GrF\to GrG$ and that this $\V$-functor is compatible with the opfibrations $p\colon GrF\to B_\V$ and $q\colon GrG\to B_\V$ (i.e.~$qa =p$). However, we must check that $a$ is actually opfibered.

 		By Proposition \ref{lem:opfibered-cocart}, it suffices to show that $a$ sends chosen $p$-opcartesian lifts to $q$-opcartesian maps.
 		Recall from (\ref{eq:Gr-chi}) that the chosen $p$-opcartesian lifts are induced by the identity maps $1_{F_f x}$ for $(x,b) \in Gr F$ and $f \colon b \to c$ in $B$.
 		Since $\alpha_b$ is a $\V$-functor, these are sent to identity maps $1_{\alpha_b F_f x}$, which are precisely the chosen $q$-opcartesian lifts in $Gr G$.
 		By Lemma \ref{lem:iso-to-chosen}, the composite of $1_{\alpha_b F_f x}$ with the  isomorphism $(\alpha_{f})_x$ remains $q$-opcartesian, as desired.

 	Now it remains to show that $Gr$ takes modifications to natural transformations of opfibered functors. 	
	Let $\Gamma\colon\alpha\Tto\beta \colon F \To G$ be a modification of pseudonatural transformations.
	This includes the data of natural transformations $\alpha_b \To \beta_b \colon F_b \to G_b$ with components $\1 \to G_b(\alpha_b x, \beta_b x)$ for each $b \in B$ and $x \in F_b$.
	Composing with the isomorphism $\alpha_b x \cong G_{1_b} \alpha_b x$, we obtain maps
	\[
		\1 \to G_b(G_{1_b} \alpha_b x, \beta_b x) \hookrightarrow Gr G\big((\alpha_bx,b),(\beta_bx,b) \big)
	\]
	which are precisely the $(x,b)$-components of the desired natural transformation $Gr\Gamma \colon Gr\,\alpha\To Gr\,\beta$.
 	\end{proof}

\section{The Grothendieck Correspondence} \label{sec:grcorr}

In this section we show that the Grothendieck construction of \S \ref{sec:grcon} and the  inverse Grothendieck construction of \S \ref{sec:invgrcon} do behave as inverses when the base category $\B$ is of the form $B_\V$.
Throughout, we make the same assumptions as \S \ref{sec:grcon}, including the identification $B = (B_\V)_0$ from Remark \ref{rem:identifyBBV0}.

\subsection{$Gr \circ I$}

	We first prove some properties of opfibrations over an arbitrary $\B$, before specializing to opfibrations over $B_\V$.

	Let $p \colon \cE \to \B$ be an opfibration, and let $q \colon Gr \cE_\bullet \to (\B_0)_\V$ be the opfibration that results from applying $I$ and $Gr$ to $p$.
	The objects of $Gr \cE_\bullet$ are pairs $(e, pe)$ where $e \in \cE$, while the morphisms are
	\[
		Gr \cE_\bullet\big( (e,pe), (e', pe') \big) = \coprod_{f \in \B_0(pe, pe')} \cE_{pe'}(\lift{f}e, e').
	\]

\begin{lem}
	For each $f \in \B_0(pe, pe')$, we have
	\[
		\cE_{pe'}(\lift{f}e,e') \cong \cE_f(e,e')
	\]
	where $\cE_f(e,e')$ is defined to be the pullback:
	\[
		\label{eq:}
		\begin{tikzcd}
			\cE_f(e,e') \arrow[dr, phantom, "\lrcorner", very near start] \ar[r] \ar[d] & \cE(e,e') \ar[d, "p" ] 
			\\
			\1 \ar[r, "f"] & \B(pe,pe') 
		\end{tikzcd}	
	\]
\end{lem}
\begin{proof}
	By Definition \ref{def:opfib} and (\ref{eq:fiber-homs}), we have a composite of pullbacks:
	\begin{equation} \label{eq:GrI-Id-components}
		\begin{tikzcd}[row sep = large]
			\cE_{pe'}(\lift{f}e, e') \ar[d] \ar[r] \arrow[dr, phantom, "\lrcorner", very near start]  
			& 
			\cE(\lift{f}e,e') \ar[d, "p"'] \ar[r, "-\circ \chi{(f,e)}"] \arrow[dr, phantom, "\lrcorner", very near start]
			& 
			\cE(e,e') \ar[d, "p"]
			\\
			\1 \ar[r, "1_{pe'}"] & \B(pe',pe') \ar[r, "-\circ f"] & \B(pe, pe')
		\end{tikzcd}
	\end{equation}
	But the outer cospan is also the defining cospan for the pullback $\cE_f(e,e')$, hence these two pullbacks are isomorphic.
\end{proof}

\begin{defi}
	Let $\epsilon_p \colon Gr \cE_\bullet \to \cE$ denote the functor that sends $(e,pe)$ to $e$, and whose action on morphisms is induced by the upper horizontal composite in (\ref{eq:GrI-Id-components}):
	\[
		Gr \cE_\bullet \big( (e,pe), (e', pe') \big) = \coprod_{f \in \B_0(pe, pe')} \cE_{pe'}(\lift{f}e, e')
		\xrightarrow{\quad -\circ {\chi(f,e)} \quad} \cE(e,e').
	\]
\end{defi}

\begin{lem} \label{lem:GrI-pullback}
	The functor $\epsilon_p \colon Gr \cE_\bullet \to \cE$ fits into the pullback
	\begin{equation} \label{eq:GrI-pullback}
		\begin{tikzcd}
			Gr \cE_\bullet \ar[dr, phantom, "\lrcorner", very near start] \ar[d, "q"'] \ar[r, "\epsilon_p"] & \cE \ar[d, "p"]
			\\
			(\B_0)_\V \ar[r, "\sigma_\B"] & \B
		\end{tikzcd}
	\end{equation}
	where $\sigma$ is the counit of the adjunction (\ref{eq:CatVCat}).
\end{lem}
\begin{proof}
	Note that $(\B_0)_\V$ and  $\B$ have the same objects and $\sigma_\B$ is the identity on objects.
	Morphisms of $(\B_0)_\V$ are given by
	\[
		(\B_0)_\V(b,b') = \coprod_{f \colon \1 \to \B(b,b')} \1,
	\]
	and $\sigma_\B$ is the coproduct of the individual maps $f \colon \1 \to \B(b,b')$, i.e.\
	\[
		\sigma_\B = \coprod_{f \colon \1 \to \B(b,b')} f.
	\]
	The outer square in (\ref{eq:GrI-Id-components}) then shows that the square in (\ref{eq:GrI-pullback}) commutes.

	To see that this square is a pullback, we first note that another pullback of $p$ along $\sigma_\B$ is given by the category $\P$ with the same objects as $\cE$ and morphisms fitting into the pullback:
	\[
		\begin{tikzcd}
			\P(e,e') \ar[dr, phantom, "\lrcorner", very near start] \ar[d] \ar[r]  & \cE(e,e') \ar[d, "p"]
			\\
			\coprod\limits_{f \colon \1 \to \B(pe,pe')} \1 \ar[r, "\coprod f"] & \B(pe, pe')
		\end{tikzcd}
	\]
	Since pullbacks preserve coproduct decompositions, we obtain
	\begin{align*}
		\P(e,e') &\cong \coprod_{f \colon \1 \to \B(pe,pe')} \cE_f(e, e') \\
				  &\cong \coprod_{f \colon \1 \to \B(pe,pe')} \cE_{pe'}(\lift{f}e,e') \\
				  &= Gr \cE_\bullet(e,e'),
	\end{align*}
	where the second isomorphism is given by the previous Lemma.
	So $Gr$ is isomorphic to $\P$, and is thus also a pullback.
\end{proof}

Let $p' \colon \cE' \to \B$ be another opfibration, and let $q' = GrI(p')$. 
For each opfibered functor $k \colon \cE \to \cE'$ from $p$ to $p'$, there is an opfibered functor $GrI(k) \colon Gr \cE_\bullet \to Gr \cE_\bullet'$ from $q$ to $q'$.
On objects,  $GrI(k)$ sends $(e, pe)$ to $(ke, pe)$, while the action on morphisms is induced by the composite
\[
	\cE_{pe'}(\lift{f}e, e') \xrightarrow{\quad k \quad} \cE'_{pe'}(k \lift{f}e, ke') \cong \cE'_{pe'}(\lift{f} ke, ke'),
\]
where the isomorphism comes from (\ref{eq:alpha-f}), and satisfies:
\begin{equation} \label{eq:chi-k}
	\begin{tikzcd}
		\cE'_{pe'}(k\lift{f}e, ke') \ar[r, "\cong"] \ar[d, "- \circ k \chi{(f,e)}"'] & \cE'_{pe'}(\lift{f} ke, ke') \ar[d, "-\circ \chi{(f,ke)}"]
		\\
		\cE'(ke, ke') \ar[r, equals] & \cE'(ke, ke')
	\end{tikzcd}
\end{equation}
Let $k' \colon \cE \to \cE'$ be another opfibered functor and let $\gamma \colon k \To k'$ be a natural transformation over $\B$. 
Since $\gamma$ lies over $\B$, its components factor as
\[
	\gamma_e \colon \1 \to \cE'_{pe}(ke,k'e) \hookrightarrow \cE'(ke, k'e).
\]
The components $GrI(\gamma)_{(e, pe)}$ are then given by the composite
\[
	\1 \xrightarrow{\gamma_{e}} \cE'_{pe}(ke,k'e) \cong \cE'_{pe}(\lift{1}ke, k'e) \hookrightarrow \coprod_{f \in \B_0(pe,pe)} \cE'_{pe}(\lift{f}ke, k'e).
\]

\begin{lem} \label{lem:GrI-2-natural}
	The following diagram commutes for all opfibered functors $k, k'$ and natural transformations $\gamma$ over $\B$:
	\[
		\begin{tikzcd}[row sep = large, column sep = huge]
			Gr \cE_\bullet \ar[d, "\epsilon_p"'] \ar[r, "Gr I (k)", ""{name = U,below}, bend left = 17] \ar[r, "Gr I (k')"', ""{name = D, above}, bend right = 17] & Gr \cE_\bullet' \ar[d, "\epsilon_{p'}"] \ar[from = U, to = D, Rightarrow, "\,Gr I(\gamma)", shift right = 3]
			\\
			\cE \ar[r, "k", bend left = 17, ""{name = UU, below}] \ar[r, "k'"', bend right=17, ""{name = DD, above}] & \cE' \ar[from = UU, to = DD, Rightarrow, "\;\gamma", shift right = 2]
		\end{tikzcd}
	\]
\end{lem}
\begin{proof}
	We first verify that $\epsilon_{p'}\circ GrI(k) = k \circ \epsilon_p$ for any $k$.
	On objects, both composites send $(e, pe)$ to $ke$.
	By (\ref{eq:chi-k}) and the functoriality of $k$, the following diagram commutes,
	\[
		\begin{tikzcd}
			\cE_{pe'}(\lift{f}e, e') \ar[r, "k"] \ar[d, "- \circ \chi{(f,e)}"'] & \cE'_{pe'}(k\lift{f}e, ke') \ar[r, "\cong"] \ar[d, "- \circ k \chi{(f,e)}"] & \cE'_{pe'}(\lift{f} ke, ke') \ar[d, "-\circ \chi{(f,ke)}"]
			\\
			\cE(e,e') \ar[r, "k"] & \cE'(ke, ke') \ar[r, equals] & \cE'(ke, ke')
		\end{tikzcd}
	\]
	which implies that the corresponding diagram on morphisms commutes.

	Next, we verify that $\epsilon_{p'} \circ GrI(\gamma) = \gamma \circ \epsilon_p$ for any $\gamma$.
	By Definition \ref{def:whiskering}, $(\gamma \circ \epsilon_p)_{(e,pe)} = \gamma_{\epsilon_p(e,pe)} = \gamma_e$, while $\left(\epsilon_{p'} \circ GrI(\gamma)\right)_{(e,pe)}$ is the composite
	\[
		\1 \xrightarrow{\gamma_{e}} \cE'_{pe}(ke,k'e) \cong \cE'_{pe}(\lift{1}ke, k'e) \cong \cE'_{pe}(ke,k'e) \hookrightarrow \cE'(ke, k'e),
	\]
	which is again just $\gamma_e$.
\end{proof}

We now specialize to the case $\B = B_\V$. 
Recall from Remark \ref{rem:identifyBBV0} that we have $(B_\V)_0 = B$, so that $Gr$ and $I$ fit into the diagram:
\[
	\begin{tikzcd}
		\OpFib(B_\V) \ar[r, bend left, "I"] & \Fun(B, \Catv) \ar[l, bend left, "Gr"]
	\end{tikzcd}
\]

\begin{prop} \label{prop:IGr}
	The maps $\epsilon_p \colon Gr \cE_\bullet \to \cE$ (for opfibrations $p \colon \cE \to B_\V$) are the components of a $2$-natural isomorphism 
	\[
		\epsilon \colon Gr \circ I \Rightarrow 1_{\OpFib(B_\V)}.
	\]
\end{prop}
\begin{proof}
	Lemma \ref{lem:GrI-2-natural} is precisely the statement of $2$-naturality of $\epsilon$, while Lemma \ref{lem:GrI-pullback} implies that $\epsilon$ is an isomorphism.

	In detail, setting $\B=B_\V$ for some category $B$, we have
	\[
		\left((B_\V)_0 \right)_\V = B_\V
	\]
	so that $\sigma_{B_\V}$ is $1_{B_\V}$. 
	By Lemma \ref{lem:GrI-pullback}, each $q\colon Gr \cE_\bullet \to B_\V$ is a pullback of $p \colon \cE \to B_\V$ along an identity, hence is isomorphic to $p$ via $\epsilon_p$.
\end{proof}

\begin{rem}
	As a consequence of the isomorphism $Gr \cE_\bullet \cong \cE$ when $\cE$ is opfibered over $B_\V$, we obtain a coproduct decomposition
	\[
		\cE(e,e') \cong \coprod_{f \in B(pe,pe')} \cE_f(e,e').
	\]
	In fact, as long as pullbacks preserve coproduct decompositions, \emph{any} functor into a free $\V$-category $p \colon \cE \to B_\V$ yields such a coproduct decomposition
	simply by pulling $p$ back along $1_{B_\V}$. 
	This does not require $p$ to be an opfibration, nor $(B_\V)_0 = B$.
\end{rem}

\begin{rem}
	The results above also answer the question: when can we recover our original opfibration $p \colon \cE\to \B$ from its pseudofunctor $\cE_\bullet$?
	In other words, when can hope for an equivalence  
	\[
		\OpFib(\B) \cong \Fun(\B_0, \Catv)?
	\]
	One sees that this can happen only if the counit $\sigma_\B \colon (\B_0)_\V \to \B$ of the adjunction $(-)_\V \dashv (-)_0$ is an equivalence.
	But since we have also assumed that the unit $\iota_B$ is an equivalence, this means that $\V$ is equivalent to $\Set$.
\end{rem}	

\subsection{$I \circ Gr$}

\begin{prop} \label{prop:GrI}
	There is a $2$-natural isomorphism 
	\[
		\eta \colon 1_{\Fun(B, \Catv)} \Rightarrow I \circ Gr.
	\]
\end{prop}
\begin{proof}
	Let $F\colon B \to \Catv$ be a pseudofunctor and let $G$ be the pseudofunctor $(Gr F)_\bullet$ obtained by applying $Gr$ and $I$ to $F$.
	We need a $2$-natural family of isomorphisms (i.e.\ invertible pseudonatural transformations) $\eta_F \colon F \Rightarrow G$.

	For a fixed $F$, we will write $\eta$ instead of $\eta_F$ for brevity.
	We first produce functors $\eta_b \colon F_b \to G_b$ for each $b \in B$, and show that each $\eta_b$ is an isomorphism of categories.
	Objects of $G_b$ are of the form $(x, b)$ where $x \in F_b$, so we may take the action of $\eta_b$ on objects to be $x \mapsto (x,b)$.
	The hom-objects $G_b\big( (x,b), (y,b)\big)$ are precisely the $f = 1_b$ part of
	\[
		GrF\big( (x,b), (y,b)\big) = \coprod_{f \colon b \to b} F_b(F_f x, y).
	\]
	We then take $\eta_b$ on morphisms to be the composite
	\[
		\begin{tikzcd}
		F_b(x,y) \ar[r, "\cong"', "- \circ \xi_x"] & F_b(F_{1_b}x,y) = G_b\big((x,b),(y,b)\big).
		\end{tikzcd}
	\] 
	This yields an isomorphism of categories $\eta_b \colon F_b \cong G_b$.

	Next, for each $f \colon b \to c$, we need an invertible $\eta_f$:
	\[	
		\begin{tikzcd}
			F_b \ar[r, "F_f"]\ar[d, "\eta_b", swap] & F_b \ar[d, "\eta_b"]\\
			G_b \ar[r, "G_f", swap] \ar[ur, Rightarrow,"\eta_f", "\cong"', start anchor={north east}, end anchor={south west}, shorten <=0.7em, shorten >=0.7em]& G_b
		\end{tikzcd}
	\]
	In fact, we may take $\eta_f$ to be the identity. 
	To see this, first note that $G_f$ is given by
	$
		G_f (x,b) = (F_f x,c)
	$
	on objects and
	\[
		\begin{tikzcd}[row sep = small]
			G_b\big((x,b),(y,b)\big)\ar[d, equals] \ar[rr, dotted, "G_f"] & & G_b\big((F_f x,c),( F_f y ,c) \big) \ar[d, equals]
			\\
			F_b(F_{1_b}x, y) \ar[r, "F_f"] & F_b(F_f F_{1_b}x, F_f y) \ar[r, "\cong"] & F_b(F_{1_b} F_f x, F_f y) 
		\end{tikzcd}
	\]
	on morphisms.
	Expressing both $\eta_b F_f$ and $G_f \eta_b$ in terms of $\xi$ and $\theta$, we see that the relevant diagram is given by
	\begin{equation} \label{eq:alphaf}
		\begin{tikzcd}
			\1 \ar[rr, "\xi_{F_f x}"] \ar[d, "\xi_x"'] & & F_b(F_{1_b} F_f x, F_f x) 
			\\
			F_b(F_{1_b}x,x) \ar[r, "F_f"] & F_b(F_f F_{1_b}x,F_f x) \ar[r, "\theta({f, 1_b})^{-1}", "\cong"'] & F_b(F_f x, F_f x) \ar[u, "\theta({1_b, f})"', "\cong"]
		\end{tikzcd}
	\end{equation}
	tensored with $F_b(x,y) \xrightarrow{F_f} F_b(F_fx,F_f y)$, then applying composition in $F_b$.
	The arrow $\xi_{F_f x}$ gives rise to $\eta_b F_f$, while the composite
	$
		\theta(1_b, f) \theta(f, 1_b)^{-1} F_f \xi_x
	$
	gives rise to $G_f \eta_b$, so that $\eta_b F_f = G_f \eta_b$ if (\ref{eq:alphaf}) commutes.
	This is the case, because the following diagram commutes (both composites being equal to $1_{F_f x}$ by (\ref{eq:pseudo-unit-left}) and (\ref{eq:pseudo-unit-right}))
	\[
		\begin{tikzcd}
			\1 \ar[rr, "\xi_{F_f x}"] \ar[d, "\xi_x"'] & & F_b(F_{1_b} F_f x, F_f x) \ar[d, "\theta({1_b, f})^{-1}", "\cong"']
			\\
			F_b(F_{1_b}x,x) \ar[r, "F_f"] & F_b(F_f F_{1_b}x,F_f x) \ar[r, "\theta({f, 1_b})^{-1}", "\cong"'] & F_b(F_f x, F_f x) 
		\end{tikzcd}	
	\]
	and $\theta(1_b,f)$ and $\theta(1_b, f)^{-1}$ are mutual inverses.

	We thus have invertible pseudonatural transformations $\eta_F \colon F \To G$ for each $F \in \Fun(B, \Catv)$.
	It remains to be shown that these are part of a $2$-natural transformation $\eta \colon 1 \To I \circ Gr$.
	Let $F'\colon B \to \Catv$ be another pseudofunctor, and $G' = (Gr F')_\bullet$.
	For pseudonatural transformations $\alpha, \alpha' \colon F \To F'$ and modification $\Gamma \colon \alpha \Tto \alpha'$ we need the following diagram to commute:
	\[
		\begin{tikzcd}[row sep = large, column sep = huge]
			F \ar[d, "\eta_F"', Rightarrow] \ar[r,"\alpha", Rightarrow, bend left= 20, ""{name = U, below}] \ar[r,"\alpha'"', Rightarrow, bend right=20, ""{name = D, above}] & F' \ar[d, "\eta_{F'}", Rightarrow] \ar[from = U, to = D, symbol = \Tto, "\;\Gamma"]
			\\
			G \ar[r, "I Gr(\alpha)", Rightarrow, bend left = 20, ""{name = UU, below}] \ar[r, "IGr(\alpha')"', Rightarrow, bend right = 20, ""{name = DD, above}] \ar[from = UU, to = DD, symbol = \Tto, "\,IGr(\Gamma)", shift right = 4] & G'
		\end{tikzcd}
	\]

	We first verify that $\eta_{F'} \cdot \alpha = IGr(\alpha)\cdot \eta_F$ for any $\alpha$.
	Fixing $b \in B$, the appropriate diagram on objects (of $F_b, F'_b$ etc.) obviously commutes, while on morphisms we need the following diagram to commute:
	\[
		\begin{tikzcd}[sep = large]
			F_b(x,y) \ar[rr, "\alpha_b"] \ar[d, "- \circ \xi_x"'] & & F'_b(\alpha_b x, \alpha_b y) \ar[d, "- \circ \xi'_{\alpha_b x}"]
			\\
			F_b(F_{1_b}x, y) \ar[r, "\alpha_b"] & F'_b(\alpha_b F'_{1_b} x, \alpha_b y) \ar[r, "- \circ (\alpha_{1_b})_x"] & F'_b(F'_{1_b} \alpha_b x, \alpha_b y)
		\end{tikzcd}
	\]
	But this is precisely the second axiom in \cite[\S 1.2]{leinsterbicats} that $\alpha$ satisfies.
	Similarly, for each $f \colon b\to c$ in $B$, the relevant diagram involving $\alpha_f$ also commutes because of the first axiom in \cite[\S 1.2]{leinsterbicats} that $\alpha$ satisfies.

	Next, we verify that $\eta_{F'} \cdot \Gamma = IGr(\Gamma) \cdot \eta_F$ for any $\Gamma$.
	The components of $\Gamma$ are natural transformations $\Gamma_b \colon \alpha_b \To \alpha'_b$ for each $b \in B$, which in turn have components $\Gamma_{b,x} \colon \1 \to F'_b(\alpha_b x, \alpha'_b x)$ for each $x \in F_b$.
	The components $(\eta_{F}\cdot \Gamma)_{b,x}$ are given by the composite
	\[
		\1 \xrightarrow{\Gamma_{b,x}}  F'_b(\alpha_b x, \alpha'_b x) \cong  F'_b(F'_{1_b}\alpha_b x, \alpha'_b x) = G'_b\left( (\alpha_b x, b), (\alpha'_b x, b) \right),
	\]
	but these are exactly the components $IGr(\Gamma)_{b, (x,b)}$.
	Since $(\eta_F)_b$ sends $x \in F_b$ to $(x,b)$, we have $(IGr(\Gamma) \cdot \eta_F)_{b,x} = IGr(\Gamma)_{b, (x,b)} = (\eta_{F}\cdot \Gamma)_{b,x}$, as desired.
\end{proof}

Putting Propositions \ref{prop:IGr} and \ref{prop:GrI} together, we obtain:
\begin{thm}\label{mainthm}
	Let $\V$ satisfy the assumptions in \S \ref{sec:grcon}, and let $B$ be a category. 
	There is a $2$-equivalence
	\[
		\OpFib(B_\V) \cong \Fun(B, \Catv).
	\]
\end{thm}

	We have proved an enriched version of the Grothendieck correspondence when $\V$ satisfies the assumptions in \S \ref{sec:grcon}, which yields the classical result by Grothendieck when $\V = \Set$.

	However, this is somewhat unsatisfactory for reasons we have mentioned in the Introduction.
	First, requiring that $\1$ is terminal and $B \cong (B_\V)_0$ seems rather restrictive, ruling out examples such as $\V = \Vect$.
	Next, even when these conditions apply, such as  when $\V = \Top, \sSet$ or $\Cat$, this result really only considers opfibrations over a  `discrete' base $B_\V$.
	Subsequent work will involve removing various assumptions on $\V$ and retaining more $\V$-structure from an opfibration over a non-discrete base.

\appendix

\section{Appendix: Enriched categories and $\Catv$}
\label{sec:appendix}

We recall some basic information about (small) enriched categories (all of which can be found in \cite[Ch. 3]{riehlhct}, or \cite{kellyenriched}).
Throughout, $\V$ will denote a locally small monoidal category with monoidal product $\otimes \colon \V \times \V \to \V$ and monoidal unit $\1$.

\begin{defi}
A \emph{$\V$-category} $\C$ is the data of:
\begin{enumerate}
\item A set of objects, which we will denote by $\C$ or $Ob(\C)$, where the former is an obvious abuse of notation.
\item For every pair of objects $c,d$ in $\C$, an object $\C(c,d)$ of $\V$.
\item For every object $c$ in $\C$ a morphism $1_c \colon \1\to\C(c,c)$ in $\V$. 
\item For each triple of objects $c, d, e$ in $\C$, a morphism in $\V$, $$\circ_{c,d,e}\colon  \C(d,e)\otimes\C(c,d)\to \C(c,e).$$ We will omit subscripts on $\circ$ when it is clear from context.
\end{enumerate}
All of which causes the following diagrams to commute in $\V$:

\[
\begin{aligned}
\begin{tikzcd}
& \C(d,d)\otimes \C(c,d)\arrow[d, "\circ"]
\\
\1\otimes\C(c,d)\arrow[ur, "1_d \otimes 1"]\arrow[r, "\cong"']  & \C(c,d)
\end{tikzcd}
\end{aligned}
\quad
\begin{aligned}
\begin{tikzcd}
\C(c,d)\otimes \C(c,c)\arrow[d, "\circ"'] &
\\ 
\C(c,d)  & \C(c,d)\otimes\1\arrow[ul, "1 \otimes 1_c"']\arrow[l, "\cong"] 
\end{tikzcd}
\end{aligned}
\]
\[
\begin{tikzcd}
\C(d,e)\otimes\C(c,d)\otimes\C(b,c)\arrow[r, "1\otimes\circ"]\arrow[d, "\circ\otimes1"'] & \C(d,e)\otimes\C(b,d)\arrow[d, "\circ"]\\ \C(c,e)\otimes \C(b,c)\ar[r, "\circ"] & \C(b,e)
\end{tikzcd}
\]
\end{defi}

Each $\V$-morphism $f \colon \1 \to \C(b,c)$ (i.e.\ a map in the underlying category $\C_0$) induces \emph{pre-} and \emph{post-composition} $\V$-morphisms:
	\[
		\begin{tikzcd}
			-\circ f \colon \C(c,d) \cong \C(c,d)\otimes \1 \ar[r, "1\otimes f"] & \C(c,d) \otimes \C(b,c) \ar[r, "\circ"] & \C(b,d)
			\\
			f \circ - \colon \C(a,b) \cong \1 \otimes \C(a,b) \ar[r, "f \otimes 1"] & \C(b,c) \otimes \C(a,b) \ar[r, "\circ"] & \C(a,c)			
		\end{tikzcd}
	\]
	We say that $f$ is an \emph{isomorphism} if the above composites are $\V$-isomorphisms for all $a,d \in \V$, and that $b$ and $c$ are \emph{isomorphic}.
	This is equivalent to $\C(-,b)$ and $\C(-,c)$ being isomorphic functors $\C_0 \to \V$, with $\C(-,f) = f\circ -$ the natural isomorphism between them.

\begin{defi}
\label{def:Vfunctor}
A functor of $\V$-categories, or \emph{$\V$-functor}, $F\colon \C\to \D$ consists of a function $F\colon  Ob(\C)\to Ob(\D)$, and for all $c,d \in \C$ a $\V$-morphism 
$$F_{c,d}\colon \C(c,d)\to \D(Fc,Fd)$$
such that the following diagrams commute in $\V$:
\[
\begin{aligned}
\begin{tikzcd}
& \C(c,c) \ar[dd, "F_{c,c}"]
\\
\1 \ar[ur, "1_b"] \ar[dr, "1_{Fc}"'] &
\\
& \D(Fc,Fc)
\end{tikzcd}
\end{aligned}
\qquad
\begin{aligned}
\begin{tikzcd}
\C(d,e)\otimes \C(c,d) \ar[r, "\circ"] \ar[dd, "F_{d,e} \otimes F_{c,d}"'] & \C(c,e) \ar[dd, "F_{c,e}"]
\\
& \phantom{\1}
\\
\D(Fd, Fe) \otimes \D(Fc,Fd) \ar[r,"\circ"] & \D(Fc,Fe)
\end{tikzcd}
\end{aligned}
\]

When it is clear from context, we may omit the subscripts in $F_{c,d}$, and use $F$ for the functor $\C \to \D$, the function $Ob(\C) \to Ob(\D)$ and the $\V$-morphism $\C(c,d) \to \D(Fc,Fd)$.
\end{defi}

\begin{defi}
Let $F,G\colon \C\to \D$ be $\V$-functors. A \emph{natural transformation} of $\V$-functors $\alpha\colon  F\Rightarrow G$ is a family of $\V$-morphisms $\alpha_b\colon \1\to \D(Fc,Gc)$ for each $c \in \C$ such that the following diagram commutes in $\V$:
\[
	\begin{tikzcd}[sep = large]
		\C(c,d) \ar[r, "F"] \ar[d, "G"'] & \D(Fc,Fd) \ar[d, "\alpha_d \circ -"]
		\\
		\D(Gc, Gd) \ar[r, "- \circ \alpha_b"] & \D(Fc, Gd)
	\end{tikzcd}
\]
\end{defi}

\begin{defi}\label{def:whiskering}
	Given $\V$-functors $F,G,H,K$, and a natural transformation $\alpha$ fitting into the following diagram,
	\[	
		\begin{tikzcd}
			\B \ar[r, "H"]
			&
			\C \ar[r, bend left, "F", ""{name = U, below}] \ar[r, bend right, "G"', ""{name = D}] \ar[from = U, to = D, Rightarrow, "\alpha"]
			&
			\D \ar[r, "K"]
			&
			\cE
		\end{tikzcd}
	\]
	let $K \circ \alpha \circ H \colon KFH \To KGH$ (or simply $K \alpha H$) denote the natural transformation whose components for each $b \in \B$ are given by the composite
	\[
		\begin{tikzcd}
			(K \alpha H)_b \colon
			\1 \ar[r, "\alpha_{Hb}"]
			&
			\D(FHb, GHb) \ar[r, "K"]
			&
			\cE(KFHb, KGHb)
		\end{tikzcd}
		.
	\]	
	This process is known as \textbf{whiskering}, and $K \alpha H$ is the \textbf{whiskered composite} of $K$, $\alpha$ and $H$.
	If $K$ (resp.\ $H$) is the identity, we write $\alpha H$ (resp.\ $K \alpha$) for the corresponding whiskered composite.	
\end{defi}

\begin{defi}
Let $\alpha\colon  F \To G$ and $\beta\colon  G \To H$ be natural transformations, where $F,G,H\colon  \C \to \D$. Their \textbf{vertical composite} is denoted $\beta \cdot \alpha\colon  F \To H$, and has components $(\beta \cdot \alpha)_b$ given by
\[
	\1 \cong \1 \otimes \1 \xrightarrow{\beta_b \otimes \alpha_b} \D(Gc,Hc) \otimes \D(Fc,Gc) \xrightarrow{\circ} \D(Fc,Hc).
\]
\end{defi}

\begin{defi}
Given $\alpha \colon F \To G$ and $\beta \colon J \To K$ as follows
\[
	\begin{tikzcd}
		\B \ar[r, bend left, "F", ""{name = U1, below}] \ar[r, bend right, "G"', ""{name = D1}] \ar[from = U1, to = D1, Rightarrow, "\alpha"]
		&
		\C \ar[r, bend left, "J", ""{name = U2, below}] \ar[r, bend right, "K"', ""{name = D2}] \ar[from = U2, to = D2, Rightarrow, "\beta"]
		&
		\D
	\end{tikzcd}
\]
their \textbf{horizontal composite} $\beta \circ \alpha \colon JF \To KG$, or simply $\beta \alpha$, is the composite $(\beta G) \cdot (J\alpha)$, or equivalently, the composite $(K \alpha) \cdot (\beta F) $.
\end{defi}

\begin{defi}
	Let $\Catv$ denote the strict 2-category of $\V$-categories, $\V$-functors, and natural transformations.
\end{defi}

\addcontentsline{toc}{section}{References}
\bibliography{references}

\end{document}